\documentclass[review]{elsarticle}
\usepackage{amsmath}
\usepackage{amsfonts}
\usepackage{amsfonts}
\usepackage{amssymb,latexsym}
\usepackage{amsthm}
\usepackage{lineno,hyperref}

\newtheorem{theorem}{\bf Theorem}[section]

\newtheorem{proposition}{\bf Proposition}[section]
\newtheorem{lemma}{\bf Lemma}[section]

\usepackage{graphicx}
\usepackage{chngcntr}
\counterwithin{figure}{section}
\counterwithin{equation}{section}

\begin{document}

\begin{frontmatter}

\title{The well-posedness, blow-up and travelling waves for a two-component Fornberg-Whitham system\tnoteref{mytitlenote}}
\tnotetext[mytitlenote]{This work was supported in part by NSFC(No.11571057).}

\author{Fei Xu}
\author{Yong Zhang}
\author{Fengquan Li\fnref{*}}

\fntext[*]{E-mail address: fqli@dlut.edu.cn}

\address{Dalian University of Technology, Dalian 116024,
People's Republic of China}

\begin{abstract}
In this paper, the two-component Fornberg-Whitham system is studied. We firstly investigate the well-posedness in classical Sobolev Space and show a blow-up scenario by local-in-time a priori estimates, then we present some sufficient conditions on the initial data to lead to wave breaking. Furthermore, we establish analytically the existence of periodic travelling waves.
\end{abstract}

\begin{keyword}
well-posedness, ~blow-up, ~wave breaking, ~travelling waves
\MSC[2010] 35Q35, 35J25, 35J60
\end{keyword}

\end{frontmatter}

\section{\bf Introduction}
In this paper we consider the following  Fornberg-Whitham system
 \begin{equation}\label{e11}
 \left\{ \begin{array}{ll}
 u_{t}+u\partial_{x}u=\partial_{x}(I-\partial_{x}^{2})^{-1}(\overline{\rho}-u), & t>0, x\in {\rm R}, \\
 \overline{\rho}_{t}+(\overline{\rho}u)_{x}=0,  & t>0, x\in {\rm R},
 \end{array} \right.
 \end{equation}
 where the variable $u(x, t)$ describes the horizontal velocity of the fluid and
the variable $\overline{\rho}(x, t)$ is in connection with the horizontal deviation of the surface from equilibrium. This system is motivated by the generation of the two-component Camassa-Holm equation in \cite{Ace} and the two-component Degasperis Procesi equation in \cite{An}. This was the vision initially proposed by Fan in \cite{Abd} which generalized the  Fornberg-Whitham equation to the two component Fornberg-Whitham system.  In \cite{Abd}, bifurcations of the traveling wave solution were studied, where soliton solution, kink solution, antikink solution and periodic solutions were presented by numerical simulation. This arouses our interest in researching the well-posedness, wave breaking phenomenon as well as travelling solutions of (\ref{e11}) in a mathematical point of view.

For $\overline{\rho}(x, t)=0$, (\ref{e11}) would be reduced to the Fornberg-whitham equation:
$$
  u_{t}-u_{xxt} + u_{x}+ uu_{x}= uu_{xxx}+3u_{x}u_{xx}.
$$
 This equation was derived by B. Fornberg and G.B. Whitham as a model to study the qualitative behaviors of wave-breaking in \cite{Bar}. In \cite{Boc}, the authors gave the rigorous proof of wave breaking for the kind of equation. In
 \cite{Col}, several blow-up phenomena of the Fornberg-whitham equation on line $R$ and on circle $T$ are established. Well-posedness in $H^{s}$ $(s>\frac{3}{2})$ for the FW equation has been established in
 \cite{Chen} by applying Kato's semigroup approach. \cite{Cru} employs a Galerkin type approximation argument
 showing that its Cauchy problem is well-posed in Sobolev spaces $H^{s}=H^{s}(T)$
 $(s>\frac{3}{2})$.
 The existence of global attractor to the viscous Fornberg-Whitham equation was proved in \cite{Die}.

 This equation can be written in the following non-local form
 \begin{equation}\label{e12}
  \partial_{t}u+uu_{x}=(I-\partial_{x}^{2})^{-1}\partial_{x}u.
 \end{equation}
 We mention that the FW equation allows for traveling wave solutions \cite{Lin}. Its peaked travelling wave solutions are of the form
  \begin{equation*}
  u(x,t)=Ae^{-\frac{1}{2}|x-\frac{4}{3}t|},
 \end{equation*}
 which were studies in \cite{Bar}.
 It is easy to see that it belongs to the family of nonlinear wave equations
  \begin{equation}\label{e13}
   \partial_{t}u+ \alpha uu_{x}= \mathcal{L}(u, u_{x}),
  \end{equation}
 which has been studied by many authors. With $\alpha=1$ and $\mathcal{L}(u)= \frac{1}{6}\partial^{3}u$ in (\ref{e13}), it becomes the well-known KdV equation \cite{Dal}. For $\alpha=1$ and $\mathcal{L}(u)=(I-\partial_{x}^{2})^{-1}\partial_{x}(u^{2}+ \frac{1}{2}u_{x}^{2})$, it becomes the Camassa-Holm (CH) equation, which is obtained by using an asymptotic expansion directly in the Hamiltonian for Eurler's equations in the shallow water regime \cite{Fra}\cite{Con}. In fact this equation was derived earlier by Fuchssteiner and Fokas (see\cite{Deg}\cite{Deg1}) as a bi-Hamiltonian generalization of KdV and authors in \cite{RCa}\cite{Fco} discovered that the equation has peaked solitons. \cite{Fra} shows the CH equation is integrable and has infinity conservation law. Based on its nice properties, Constantin and Escher show that the CH equation not only has global solutions but also blow-up solutions in finite time \cite{Fza,Boc}. In \cite{Guan}, Guan and Yin proved the global existence and blow-up of two component Camassa-Holm equation. If $\alpha=1$ and $\mathcal{L}(u)= -\frac{3}{2}(I-\partial_{x}^{2})^{-1}\partial_{x}(u^{2})$ in (\ref{e13}), it becomes the Degasperis-Procesi (DP) equation, which was deduced by Degasperis and Procesi \cite{ADp} and admits peakon solutions \cite{DH}. At the same time, some results on two-component Degasperis-Procesi equation are also investigated in \cite{An}.

 To the best of our knowledge, there are less analytic results on the two-component Fornberg-Whitham system except some numerical investigations in \cite{Abd}. In fact, unlike the special structure of the CH equation,
 the FW equation is not integrable, whose solutions' oddness (or evenness) will not be guaranteed by the initial data's and its solutions may change signs even though the initial data does not change signs. Moreover, the useful conservation laws
$$
\int_{R} u^{2}(x)dx =\int_{R} u^{2}_{0}(x)dx
$$
  for eq. (\ref{e12}) is no longer applicable to the case of two-component. Thus, for the system (\ref{e11}), it is more challenging to prove local well-posedness of strong solutions and to analyze the sufficient and necessary conditions of global existence and blow-up phenomena. The goal of this paper is to establish related results and investigate analytically the existence of periodic travelling wave solutions to the FW system. We mention here, despite the equivalence between blow up and wave breaking established in Theorem \ref{th41} and upper bound of $u_{x}$ got in Lemma \ref{lem44}, the lower bound of $u_{x}$ is still not addressed well. Thus the global existence of (\ref{e11}) and (\ref{e12}) are worthy of further study.

The paper is organized as follows: In Section 2, we state some commutator estimates and recall properties of the mollifier. Section 3 is devoted to the well-posedness of the system (\ref{e11}) with initial data $(u_{0}, \rho_{0})\in H^{s}\times H^{s-1}$, $s > \frac{3}{2}$. In section 4, a useful prior estimate for solutions and the supremum of the slope of the solutions are given. Then we present a precise blow-up scenario and investigate wave breaking of (\ref{e31}) corresponding to a large class of initial data. Finally, Section 5 is dedicated to establish analytically travelling wave solutions.

\section{Preliminaries}
Let $\Lambda= (I-\partial_{x}^{2})^{\frac{1}{2}}$ so that for any test function $f$, we have $\mathcal{F}(\Lambda^{s}f)=
(1+k^{2})^{\frac{s}{2}}\widehat{f}(k)$.
And we define the commutator
 \begin{equation}\label{e21}
  [\Lambda^{s}, f]=\Lambda^{s}f-f\Lambda^{s},
 \end{equation}
 in which a text function $f$ is regarded as a multiplication operator. Then, we have the following basic estimates.
 \begin{proposition}\label{p21} \cite{Tay}
 If $s>\frac{3}{2}$, $r+1\geq 0$ and $r\leq s-1$, then
 \begin{equation}\label{e22}
  \|[\Lambda^{r}\partial_{x}, f]g\|_{L^{2}}\leq c_{s,r}\|f\|_{H^{s}}\|g\|_{H^{r}}.
 \end{equation}
 \end{proposition}
 \begin{proposition}\cite{Kato}\label{p22}
 If $s\geq 0$, then
 \begin{equation}\label{e23}
  \|[\Lambda^{s}, f]g\|_{L^{2}}\leq c_{s}\|\partial_{x}f\|_{L^{\infty}}\|\Lambda^{s-1}g\|_{L^{2}}+\|\Lambda^{s}f\|_{L^{2}}\|g\|_{L^{\infty}}.
 \end{equation}
 \end{proposition}
 \begin{proposition}\label{p23}\cite{Tay}
  let $J_{\varepsilon}$ be the mollifier defined above, and $f$, $g$ be two test function, then
 \begin{equation}\label{e24}
  \|[J_{\varepsilon}, f]g\|_{L^{2}}\leq c\|f\|_{Lip}\|g\|_{H^{-1}}.
 \end{equation}
 \end{proposition}
\begin{proposition}\label{p24}\cite{TAY}
  For $s\geq 0$
 \begin{equation}\label{e25}
  \|fg\|_{H^{s}}\leq C (\|f\|_{\infty}\|g\|_{H^{s}}+\|g\|_{\infty}\|f\|_{H^{s}}).
 \end{equation}
 \end{proposition}
 \begin{lemma}\label{lem21}
 let $f$ be any test function, and $\sigma \in R$, then
  $$\|\Lambda^{\sigma}f\|_{L^{2}}= \|f\|_{H^{\sigma}},~~~~\|(I-\partial_{x})^{2}f\|_{H^{\sigma}}= \|f\|_{H^{\sigma-2}}, ~~~~~~\|\partial_{x}f\|_{H^{\sigma}}= \|f\|_{H^{\sigma+1}}.$$
\end{lemma}
 \begin{lemma}\label{lem22}
  For $r \leq s$ we have
  \begin{equation}\label{e26}
  \|I-J_{\varepsilon}\|_{\mathcal{L}(H^{s}; H^{r})}= o(\varepsilon^{s-r}).
 \end{equation}
 Also, for any test function $f$, we have for all $s > 0$, $J_{\varepsilon}f\rightarrow f\in H^{s}$.
 \end{lemma}

\begin{lemma}(growth estimate)\label{lem23}
Let $r > s$, then for any test function $f$
\begin{equation}\label{e27}
  \|J_{\varepsilon}f\|_{H^{r}}\lesssim \varepsilon^{s-r}\|f\|_{H^{s}}.
 \end{equation}
\end{lemma}

\begin{lemma}(Sobolev interpolation lemma)\label{lem24}
Let $s_{0}< s < s_{1}$ be real numbers, then
\begin{equation}\label{e28}
  \|f\|_{H^{s}}\leq \|f\|_{H^{s_{0}}}^{\frac{s_{1}-s}{s_{1}-s_{0}}} \|f\|_{H^{s_{1}}}^{\frac{s-s_{0}}{s_{1}-s_{0}}}.
 \end{equation}
\end{lemma}

\section{Local Well-posedness}
  We mention that in the hydrodynamical derivation of (\ref{e11}), it is required that $u(x, t)\rightarrow 0$ and
$\overline{\rho}(x, t)\rightarrow 1$ as $|x|\rightarrow \infty $ at any instant $t$. Then, letting $\rho = \overline{\rho}-1$, we have $\rho(x, t)\rightarrow 0$ as $|x|\rightarrow \infty $.
  We consider the initial value problem for the Fornberg-Whitham system
 \begin{equation}\label{e31}
 \left\{ \begin{array}{ll}
 u_{t}+u\partial_{x}u=\partial_{x}(I-\partial_{x}^{2})^{-1}(\rho-u), \\
 \rho_{t}+u\partial_{x}\rho+\rho\partial_{x}u+\partial_{x}u=0, \\
 u(0, x)=u_{0};~~ \rho(0, x)= \rho_{0}.
 \end{array} \right.
 \end{equation}

Our well-posedness result is stated as follows
\begin{theorem}\label{th31}
 If $s > \frac{3}{2}$ and $ (u_{0}, \rho_{0} )\in H^{s}\times H^{s-1}$, then there exists a lifetime $ \widetilde{T }> 0$ and a unique solution $ (u, \rho)\in C([0, \widetilde{T}]; H^{s}\times H^{s-1})$ of the initial value problem, which depends continuously on the initial data.
 \begin{equation*}
 \|(u, \rho)(t)\|_{H^{s}\times H^{s-1}}^{2}\leq K\|( u_{0}, \rho_{0})\|_{H^{s}\times H^{s-1}}^{2}
 \end{equation*}
for all $t\in[0,\widetilde{T}]$, where $K$ is a constant independent of $\varepsilon$.
 \end{theorem}
  To prove well-posedness, we employ a Galerkin approximation argument, and the proof closely resembles the ideas \cite{Him} .
The strategy will be to mollify the nonlinear terms in the system to construct a family of ODEs.

For $0<\varepsilon\leq 1$, let $J_{\varepsilon}$ be a standard mollifier based on some smooth and compactly supported function $\rho $ on $R$. We apply the mollifier $J_{\varepsilon}$ to the FW equation system to construct a family of ODEs in $H^{s}$
\begin{equation}\label{e32}
  \left\{
  \begin{aligned}
  &\partial_{t}u_{\varepsilon}+J_{\varepsilon}(J_{\varepsilon}u_{\varepsilon}J_{\varepsilon}\partial_{x}u_{\varepsilon})=
  \partial_{x}(I-\partial_{x})^{-1}(\rho_{\varepsilon}-u_{\varepsilon}),  \\
 & \partial_{t}\rho_{\varepsilon}+\partial_{x}J_{\varepsilon}(J_{\varepsilon}u_{\varepsilon}J_{\varepsilon}\rho_{\varepsilon})+
  \partial_{x}u_{\varepsilon}=0. \\
\end{aligned}
 \right.
\end{equation}
The equation (\ref{e32}) is equivalent to the following  regularised problem
\begin{equation}\label{e33}
\partial_{t}U_{\varepsilon}+J_{\varepsilon}[J_{\varepsilon}(A(U))\partial_{x}(J_{\varepsilon}U_{\varepsilon})]+
B(V)\partial_{x}(I-\partial_{x}^{2})^{-1}U_{\varepsilon}=0
\end{equation}
with initial data $U_{\varepsilon}(0,x)=U_{0}(x)=(u_{0}^{\varepsilon}, \rho_{0}^{\varepsilon})$.

or, with
$$\begin{matrix}
U=\begin{bmatrix}
u\\
\rho
\end{bmatrix}& A(U)=\begin{bmatrix}
u&0\\
\rho+1& u
\end{bmatrix}&B(U)=\begin{bmatrix}
-1&1\\
0&0
\end{bmatrix}
\end{matrix}$$

 We can apply the Banach ODE theorem, for each $ 0 <\varepsilon\leq 1$, there exists a lifespan $T_{\varepsilon} > 0$ such that there is a unique solution $U_{\varepsilon}\in ([0, T_{\varepsilon}]; H^{s})$. For the case $ s \geq 2$, we can refer
 to \cite{Pei}. Here, we consider $U_{\varepsilon}(0,x)=U_{0}(x)=(u_{0}, \rho_{0})\in H^{s}\times H^{s-1}$, $s > \frac{3}{2}$, and there will be a common lifespan $T$, independent of $\varepsilon$.
 Our first step is to develop a priori estimates for regularised equation (\ref{e32}).

\subsection{Lifespan and energy estimate}
Applying the $ \Lambda^{s}$ to the first equation of (\ref{e32}) and multiplying by $\Lambda^{s}u_{\varepsilon}$, then integrating on $R$, we obtain
\begin{align}\label{a34}
\frac{1}{2}\frac{d}{dt}\| u_{\varepsilon}\|_{H^{s}}^{2}+\int_{R}\Lambda^{s}J_{\varepsilon}(J_{\varepsilon}u_{\varepsilon}J_{\varepsilon}
\partial_{x}u_{\varepsilon})\Lambda^{s}u_{\varepsilon}dx&
=\int_{R}\partial_{x}(I-\partial_{x}^{2})^{-1}\Lambda^{s}
\rho_{\varepsilon}\Lambda^{s}u_{\varepsilon}dx   \nonumber \\
&-\int_{R}\partial_{x}(I-\partial_{x}^{2})^{-1}\Lambda^{s}
u_{\varepsilon}\Lambda^{s}u_{\varepsilon}dx.
\end{align}
A simple computation will show the second term to the right of the equals sign is zero by the fundamental theorem of calculus. Therefore, the estimate reduces to
\begin{align}\label{a35}
\frac{1}{2}\frac{d}{dt}\| u_{\varepsilon}\|_{H^{s}}^{2}
&\leq |\int_{R}\Lambda^{s}J_{\varepsilon}(J_{\varepsilon}u_{\varepsilon}J_{\varepsilon}
\partial_{x}u_{\varepsilon})\Lambda^{s}u_{\varepsilon}dx |+ |\int_{R}\partial_{x}(I-\partial_{x}^{2})^{-1}\Lambda^{s}
\rho_{\varepsilon}\Lambda^{s}u_{\varepsilon}dx |    \nonumber\\
&\leq |\int_{R}\Lambda^{s}(J_{\varepsilon}u_{\varepsilon}J_{\varepsilon}\partial_{x}u_{\varepsilon})
\Lambda^{s}J_{\varepsilon}u_{\varepsilon}dx|+
|\int_{R}\partial_{x}(I-\partial_{x}^{2})^{-1}\Lambda^{s}
\rho_{\varepsilon}\Lambda^{s}u_{\varepsilon}dx | . \nonumber\\
\end{align}
We use Proposition \ref{p22} and Sobolev embedding theorem to obtain
\begin{align}\label{a36}
&\mid\int_{R}\Lambda^{s}(J_{\varepsilon}u_{\varepsilon}J_{\varepsilon}\partial_{x}u_{\varepsilon})
\Lambda^{s}J_{\varepsilon}u_{\varepsilon}dx\mid \nonumber \\ &\leq\mid\int_{R}[\Lambda^{s},J_{\varepsilon}u_{\varepsilon}]J_{\varepsilon}\partial_{x}u_{\varepsilon}
\Lambda^{s}J_{\varepsilon}u_{\varepsilon}dx\mid+\mid\int_{R}J_{\varepsilon}u_{\varepsilon}\Lambda^{s}
J_{\varepsilon}\partial_{x}u_{\varepsilon}\Lambda^{s}J_{\varepsilon}u_{\varepsilon}dx\mid \nonumber \\
&\leq\|[\Lambda^{s},J_{\varepsilon}u_{\varepsilon}]J_{\varepsilon}\partial_{x}u_{\varepsilon}\|_{L^{2}}
\|\Lambda^{s}J_{\varepsilon}u_{\varepsilon}\|_{L^{2}}+\frac{1}{2}\mid\int_{R}
J_{\varepsilon}\partial_{x}u_{\varepsilon}\Lambda^{s}
J_{\varepsilon}u_{\varepsilon}\Lambda^{s}J_{\varepsilon}u_{\varepsilon}dx\mid \nonumber\\
&\leq c_{s}\parallel J_{\varepsilon}\partial_{x}u_{\varepsilon}\parallel_{\infty}\|\Lambda^{s}J_{\varepsilon}
u_{\varepsilon}\|_{L^{2}}^{2}+\frac{1}{2} \| J_{\varepsilon}\partial_{x}u_{\varepsilon}\|_{\infty}\|\Lambda^{s}J_{\varepsilon}u_{\varepsilon}
\|_{L^{2}}^{2}\leq( c_{s}+\frac{1}{2})c\| J_{\varepsilon}u_{\varepsilon}\|_{H^{s}}^{3},
\end{align}
where $c$ is the Soblev constant and
\begin{align}\label{a37}
|\int_{R}\partial_{x}(I-\partial_{x}^{2})^{-1}\Lambda^{s}
\rho_{\varepsilon}\Lambda^{s}u_{\varepsilon}dx |
\leq\|\partial_{x}(I-\partial_{x}^{2})^{-1}\Lambda^{s}\rho_{\varepsilon}\|_{L^{2}}
\|\Lambda^{s}u_{\varepsilon}\|_{L^{2}}\leq \|\rho_{\varepsilon}\|_{H^{s-1}}\| u_{\varepsilon}\|_{H^{s}}.
\end{align}
Thus (\ref{a36}) and (\ref{a37}) imply
\begin{align}\label{a38}
\frac{1}{2}\frac{d}{dt}\| u_{\varepsilon}\|_{H^{s}}^{2}\leq (c_{s}+\frac{1}{2})c\| J_{\varepsilon}u_{\varepsilon}\|_{H^{s}}^{3}+\|\rho_{\varepsilon}\|_{H^{s-1}}\|u_{\varepsilon}\|_{H^{s}}.
\end{align}
Applying the $ \Lambda^{s-1}$ to the left hand side of the second equation of (\ref{e31}) and multiplying by $\Lambda^{s-1}\rho_{\varepsilon}$ on the right hand side, then integrating on $R$, we obtain
\begin{align}\label{a39}
&\frac{1}{2}\frac{d}{dt}\| \rho_{\varepsilon}\|_{H^{s-1}}^{2}\nonumber \\
&= -\int_{R}\Lambda^{s-1}\partial_{x}(J_{\varepsilon}\rho_{\varepsilon}J_{\varepsilon}u_{\varepsilon})
\Lambda^{s-1}J_{\varepsilon}\rho_{\varepsilon}dx-\int_{R}\Lambda^{s-1}
\partial_{x}u_{\varepsilon}\Lambda^{s-1}\rho_{\varepsilon}dx \nonumber \\
&\leq\|[\Lambda^{s-1}\partial_{x}, J_{\varepsilon}u_{\varepsilon}]J_{\varepsilon}\rho_{\varepsilon}\|_{L^{2}}\|\Lambda^{s-1}\rho_{\varepsilon}\|_{L^{2}}
+\|J_{\varepsilon}u_{\varepsilon}\|_{L^{\infty}}\|\Lambda^{s-1}\partial_{x}J_{\varepsilon}u_{\varepsilon}\|_{L^{2}}
\|\Lambda^{s-1}\rho_{\varepsilon}\|_{L^{2}}+\|\Lambda^{s-1}\partial_{x}J_{\varepsilon}u_{\varepsilon}\|_{L^{2}}
\|\Lambda^{s-1}\rho_{\varepsilon}\|_{L^{2}}\nonumber\\
&\leq c_{s}\| J_{\varepsilon}u_{\varepsilon}\|_{H^{s}}\| J_{\varepsilon}\rho_{\varepsilon}\|_{{H^{s-1}}}^{2}+c\| J_{\varepsilon}u_{\varepsilon}\|_{{H^{s}}}^{2}\| J_{\varepsilon}\rho_{\varepsilon}\|_{H^{s-1}}+
\| J_{\varepsilon}u_{\varepsilon}\|_{{H^{s}}}\|J_{\varepsilon}\rho_{\varepsilon}\|_{H^{s-1}}\nonumber\\
&\lesssim \|u_{\varepsilon}\|_{{H^{s}}}\|\rho_{\varepsilon}\|_{H^{s-1}}^{2}+\| u_{\varepsilon}\|_{{H^{s}}}^{2}\| \rho_{\varepsilon}\|_{H^{s-1}}+ \|u_{\varepsilon}\|_{{H^{s}}}\|\rho_{\varepsilon}\|_{H^{s-1}},
\end{align}
where we use proposition \ref{p21} and proposition \ref{p22}. By (\ref{a38}) and (\ref{a39}) we can get
\begin{align}\label{a310}
&\frac{1}{2}\frac{d}{dt}(\| u_{\varepsilon}\|_{H^{s}}^{2}+\| \rho_{\varepsilon}\|_{H^{s-1}}^{2})\nonumber\\
&\lesssim\parallel u_{\varepsilon}\|_{H^{s}}^{3}+\|\rho_{\varepsilon}\|_{H^{s-1}}\| u_{\varepsilon}\|_{H^{s}}+\| u_{\varepsilon}\|_{{H^{s}}}\| \rho_{\varepsilon}\|_{H^{s-1}}^{2}+\| u_{\varepsilon}\|_{{H^{s}}}^{2}\| \rho_{\varepsilon}\|_{H^{s-1}}\nonumber \\
&\lesssim\frac{1}{2}\| u_{\varepsilon}\|_{{H^{s}}}^{2}+\frac{1}{2}\| \rho_{\varepsilon}\|_{H^{s-1}}^{2}+\frac{1}{2}
\| u_{\varepsilon}\|_{{H^{s}}}^{4}+\frac{1}{2}\| \rho_{\varepsilon}\|_{H^{s-1}}^{4}.
\end{align}
Let us define the energy $E^{s}(u,\rho)$ by
\begin{equation}\label{e311}
E^{s}(u,\rho)=\frac{1}{2}\| u_{\varepsilon}\|_{H^{s}}^{2}+\frac{1}{2}\| \rho_{\varepsilon}\|_{H^{s-1}}^{2}.
\end{equation}
 We get the energy estimate from (\ref{a310})
\begin{equation}\label{e312}
 \frac{d}{dt}E^{s}(u,\rho)\lesssim E^{s}(u,\rho)+E^{s}(u,\rho)^{2}.
\end{equation}
Let $y(t):=E^{s}(u, \rho)(t)$ denote the energy defined in (\ref{e311}). Then (\ref{e312}) leads to the inequality
$y'(t)\leq C(y(t)+y^{2}(t))$, which can be integrated to obtain
\begin{equation*}
y(t)(1-\frac{y_{0}}{1+y_{0}}e^{Ct})\leq \frac{y_{0}}{1+y_{0}}e^{Ct}.
\end{equation*}
 We choose $K \geq 1$, satisfies
 \begin{equation}\label{e313}
  \left\{
  \begin{aligned}
  & \frac{y_{0}}{1+y_{0}}e^{Ct}< 1,  \\
  & \frac{\frac{e^{Ct}}{1+y_{0}}}{1-\frac{y_{0}}{1+y_{0}}e^{Ct}}\leq K, \\
\end{aligned}
 \right.
\end{equation}
and we choose $ \widetilde{T}\leq T_{1}=\frac{1}{C_{1}}$log$(1+\frac{C_{1}-1}{1+C_{1}\|( u_{0}), \rho_{0},\|_{H^{s+1}\times H^{s}}})$,
where $C_{1}$ is large enough. We can get
\begin{equation}\label{e314}
\frac{y_{0}}{1+y_{0}}e^{ct}\leq \frac{1+Ky_{0}}{K(1+y_{0})}< 1~~~~ and ~~~~
\|(u_{\varepsilon}, \rho_{\varepsilon})(t)\|_{H^{s}\times H^{s-1}}^{2}\leq K\|( u_{0}, \rho_{0})\|_{H^{s}\times H^{s-1}}^{2}
\end{equation}
for all $t\in[0,\widetilde{T}]$.

\subsection{Convergence}

 $\mathbf{Step ~1:}$ Since $u_{\varepsilon}$ and $\rho_{\varepsilon}$ are uniformly bounded in $L^{\infty}([0, \widetilde{T}]; H^{s}(R))$ and $L^{\infty}([0, \widetilde{T}]; H^{s-1}(R))$ respectively. For fixed $t$, there is a subsequence such that
 \begin{align}
 &u_{\varepsilon_{j}}(t)\rightharpoondown u(t) ~~weakly~~ in ~~H^{s}(R)~~ for ~~\varepsilon_{j} \rightarrow 0,  \nonumber\\
 &u_{\varepsilon_{j}}(t)\rightarrow u(t) ~~strongly~~ in ~~ L^{2}_{loc}(R), \nonumber
 \end{align}
  and
 \begin{equation*}
 u_{\varepsilon_{j} }(t)\rightarrow u(t)~~ a.e.~~ on ~ R~~as ~~\varepsilon_{j} \rightarrow 0.
 \end{equation*}
 By Helly's theorem, we have
 \begin{equation*}
 u_{\varepsilon_{j} }\rightarrow u~~ a.e.~~ on ~[0,T]\times R~~as ~~\varepsilon_{j} \rightarrow 0
 \end{equation*}
 and
 \begin{equation*}
 \|u(t)\|_{H^{s}(R)}\leq \mathop{liminf}\limits_{\varepsilon_{j}\rightarrow 0}\|u_{\varepsilon_{j}}(t)\|_{H^{s}(R)}\leq K E^{s}(U_{0})
 \end{equation*}
 for a.e. $t\in [0, \widetilde{T}]$. The previous relation implies that $u\in L^{\infty}([0, \widetilde{T}]; H^{s}(R))$. Since $s-1 > \frac{1}{2}$, the similar argument can be applied to $\rho\in L^{\infty}([0, \widetilde{T}]; H^{s-1}(R)) $, i.e.
  \begin{equation}\label{e315}
  \rho_{\varepsilon_{j}}\rightarrow \rho~~ a.e.~~ on ~[0, \widetilde{T}]\times R~~as ~~\varepsilon_{j}\rightarrow 0.
 \end{equation}

$\mathbf{Step~2:}$ From $Young's$ inequality and equation (\ref{e31}), we infer that the the family $\{\partial_{t}u_{\varepsilon}(t,\cdot)\}$ is uniformly bounded in $H^{s-1}(R)$ as $t\in [0,\widetilde{T}]$.
 We also have an uniform bound on $\|u_{\varepsilon}(t, \cdot)\|_{H^{s}}$ for all $t\in[0,\widetilde{T}]$ due to (\ref{e314}). Let $t_{1}$ and $t_{2}\in[0,\widetilde{T}]$, the mean value theorem implies
 \begin{equation*}
\|u_{\varepsilon}(t_{1})-u_{\varepsilon}(t_{2})\|_{H^{s-1}}\leq \sup_{t\in [0,\widetilde{T}]}\|\partial_{t}u_{\varepsilon}\|_{H^{s-1}}|t_{1}-t_{2}|.
\end{equation*}
Combining above estimates and lemma \ref{lem24}, we have
 \begin{equation}\label{e316}
 \| u_{\varepsilon}(t_{1})-u_{\varepsilon}(t_{2})\|_{H^{s-\sigma}}\leq\| u_{\varepsilon}(t_{1})-u_{\varepsilon}(t_{2})\|_{H^{s-1}}^{1-\sigma}\| u_{\varepsilon}(t_{1})-u_{\varepsilon}(t_{2})\|_{H^{s}}^{\sigma} \leq K |t-s|^{1-\sigma},~~t,s\in [0,\widetilde{T}].
 \end{equation}
From (\ref{e314}) and (\ref{e316}), we apply Ascoli's theorem to conclude $u\in C([0, \widetilde{T}]; H^{s-\sigma}(R))$. We may restrict $\sigma > 0$ small so that $ s-\sigma > \frac{3}{2}$ and employ the Sobolev embedding
theorem to get $u\in C([0, \widetilde{T}]; C^{1}(R))$.

From the equation (\ref{e32}), we use Lemma \ref{lem21} and (\ref{e314}) to get
 \begin{align}\label{a317}
 \|\rho_{\varepsilon t}\|_{H^{-1}}&\leq
 \|\rho_{\varepsilon x}\|_{H^{-1}}\|u_{\varepsilon }\|_{L^{\infty}}+\|u_{\varepsilon x }\|_{H^{-1}}\|\rho_{\varepsilon} \|_{L^{\infty}}+\|u_{\varepsilon x }\|_{H^{-1}} \nonumber \\
 &\leq \|\rho_{\varepsilon}\|_{L^{2}}\|u_{\varepsilon}\|_{L^{\infty}}+\|u_{\varepsilon }\|_{L^{2}}\|\rho_{\varepsilon} \|_{L^{\infty}}+\|u_{\varepsilon}\|_{L^{2}}\leq C E^{s}(U_{0})^{2}.
 \end{align}
 From the mean-value theorem we obtain
 \begin{equation*}
 \| \rho_{\varepsilon}(t)-\rho_{\varepsilon}(s)\|_{H^{-1}}\leq K |t-s|,~~t,s\in [0,\widetilde{T}].
 \end{equation*}
 By lemma \ref{lem24}, for any $\theta\in [0,1]$, there is a $K >0$ such that
 \begin{equation}\label{e318}
 \| \rho_{\varepsilon}(t)-\rho_{\varepsilon}(s)\|_{H^{s\theta-1}}\leq K |t-s|^{1-\theta},~~t,s\geq 0.
 \end{equation}
 Given $\theta\in [0,1]$, $s\theta-1> \frac{3}{2}\theta-1$, by (\ref{e318}) we can get
 \begin{equation*}
 \|\rho_{\varepsilon}\|_{BC^{1-\theta}([0,\widetilde{T}];H^{s\theta-1})}\leq K.
 \end{equation*}
 In particular, for $\theta=1$, we have the Sobolev embedding $H^{s\theta-1}(R)\hookrightarrow C_{\infty}(R)$, we find that
 \begin{equation}\label{e319}
 \|\rho_{\varepsilon}\|_{C([0,\widetilde{T}]\times R)}\leq K.
 \end{equation}
 By (\ref{e318}) and (\ref{e319}), we may
 again apply Ascoli's theorem to conclude $\rho\in C([0,\widetilde{T}]\times R)$.

 $\mathbf{Step ~ 3:}$ Our final refinement is to show that $u\in C([0,\widetilde{T}], H^{s})$ and $\rho\in C([0,\widetilde{T}], H^{s-1})$.
 For any $t \in [0,\widetilde{T}]$, and let $t_{n}$ be a sequence which converges to $t$. It is sufficient to prove that $\mathop{lim}\limits_{t_{n}\rightarrow t}\| u(t_{n})-u(t)\|_{H^{s}}=0$, i.e., to prove that $\| u(t_{n})\|_{H^{s}}\rightarrow \| u(t)\|_{H^{s}}$
  and $(u(t_{n}), u(t))_{H^{s}(R)}\rightarrow (u(t), u(t))_{H^{s}(R)}$.

  We now verify the first assertion, which is equivalent to showing that the map $t\rightarrow \| u(t)\|_{H^{s}}$ is continuous. Let
  $$ F(t)= \| u(t)\|_{H^{s}}^{2},~~~F_{\varepsilon}(t)= \| J_{\varepsilon}u(t)\|_{H^{s}}^{2}.$$
  It's easy to see $F_{\varepsilon}(t)\rightarrow F(t)$ pointwise as
  $\varepsilon\rightarrow 0$ by definition of the molifier. We will prove that the family $F_{\varepsilon}(t)$ is Lipschitz continuous. Applying the operator $\Lambda^{s}J_{\varepsilon}$ on the left hand side of the first equation of (\ref{e32}), multiplying on right hand side by $\Lambda^{s}J_{\varepsilon}u$ and integrating on $R$, we apply Proposition \ref{p23} and Proposition \ref{p24} to obtain
 \begin{align}\label{a320}
 &\frac{1}{2}\frac{d}{dt}F_{\varepsilon}(t)=\frac{1}{2}\frac{d}{dt}\| J_{\varepsilon}u(t)\|_{H^{s}}^{2} \nonumber \\
 &=\int_{R}\Lambda^{s}J_{\varepsilon}(uu_{x})\Lambda^{s}J_{\varepsilon}udx +
 \int_{R}\Lambda^{s}J_{\varepsilon}\partial_{x}(I-\partial_{x}^{2})^{-1}(\rho_{\varepsilon}-u_{\varepsilon})\Lambda^{s}J_{\varepsilon}udx \nonumber\\
 & \leq\int_{R}[\Lambda^{s},u]u_{x}\Lambda^{s}J_{\varepsilon}^{2}udx+
 \int_{R}J_{\varepsilon}u\Lambda^{s}u_{x}\Lambda^{s}J_{\varepsilon}udx+
 \int_{R}\Lambda^{s}\partial_{x}(I-\partial_{x}^{2})^{-1}(\rho_{\varepsilon}-u_{\varepsilon})\Lambda^{s}J_{\varepsilon}^{2}udx \nonumber\\
 &\leq \int_{R}[\Lambda^{s},u]u_{x}\Lambda^{s}J_{\varepsilon}^{2}udx+
 \int_{R}[J_{\varepsilon},u]\Lambda^{s}u_{x}\Lambda^{s}J_{\varepsilon}udx+
 uJ_{\varepsilon}\Lambda^{s}u_{x}\Lambda^{s}J_{\varepsilon}udx+
 \int_{R}\Lambda^{s}\partial_{x}(I-\partial_{x}^{2})^{-1}u_{\varepsilon}\Lambda^{s}J_{\varepsilon}^{2}udx\nonumber\\
 &\leq C_{s}(\|\partial_{x}u\|_{L^{\infty}}\|\Lambda^{s-1}\partial_{x}u\|_{L^{2}}+
 \|\Lambda^{s}u\|_{L^{2}}\|\partial_{x}u\|_{L^{\infty}})\|\Lambda^{s}J_{\varepsilon}u\|_{L^{2}}+
 \| u\|_{Lip}\|\Lambda^{s}u_{x}\|_{H^{-1}}\|\Lambda^{s}J_{\varepsilon}u\|_{L^{2}}\nonumber\\
 &+\frac{1}{2}\|\partial_{x}u\|_{\infty}\|\Lambda^{s}J_{\varepsilon}u\|_{L^{2}}^{2}
 +\|\Lambda^{s}\partial_{x}(I-\partial_{x}^{2})^{-1}\rho_{\varepsilon}\|_{L^{2}}
 \|\Lambda^{s}J_{\varepsilon}u\|_{L^{2}}+\|\Lambda^{s}\partial_{x}(I-\partial_{x}^{2})^{-1}u_{\varepsilon}\|_{L^{2}}
 \|\Lambda^{s}J_{\varepsilon}u\|_{L^{2}}\nonumber \\
 &\lesssim\| u\|_{H^{s}}^{3}+\| \rho\|_{H^{s-1}}\| u\|_{H^{s}} \nonumber \\
  &\lesssim E^{s}(U_{0})^{2}+E^{s}(U_{0})
 \end{align}
Since the right hand side is bounded, independent of $\varepsilon$, we may conclude that $F_{\varepsilon}(t)$ is uniformly continuous.

Then we will verify the second assertion.
Let $\varepsilon>0$ and $\varphi \in C^{\infty}$ such
that $\|u(t)-\varphi \|_{H^{s}}\leq \frac{\varepsilon}{4KE^{s}(U_{0})\|u_{0}\|}$. By the triangle inequality,
we have
\begin{align}\label{a321}
|(u(t_{n})-u(t), u(t))_{H^{s}}| &\leq |(u(t_{n})-u(t), u(t)-\varphi)_{H^{s}}|
+| (u(t_{n})-u(t), \varphi)_{H^{s}}|.\nonumber \\
&\leq \| u(t_{n})-u(t)\|_{H^{s}}\| u(t)-\varphi\|_{H^{s}}+
| (u(t_{n})-u(t), \varphi)_{H^{s}}|.\nonumber \\
&\leq 2KE^{s}(U_{0})\frac{\varepsilon}{16\|u_{0}\|}+
\| u(t_{n})-u(t)\|_{H^{s-1}} \|\varphi\|_{H^{s+1}} \nonumber\\
&\leq \frac{\varepsilon}{2}+ K | t_{n}-t|,
\end{align}
where we use (\ref{e314}). Choosing $t_{n}$ sufficiently close to $t$, we obtain $|(u(t_{n})-u(t), u(t))_{H^{s}}|\leq \varepsilon$.
Similarly, for $\rho_{\varepsilon}$, we have similar result, thus we complete the proof of existence of a solution to
the Fornbern-Whitham system.

\subsection{Uniqueness of solutions}
 Let $ (u_{1}, \rho_{1})$, $(u_{2}, \rho_{2})$ be the solutions to the FW equation system (\ref{e31}) corresponding to the same initial data $(u_{0}, \rho_{0})$. Let
\begin{equation*}
 \omega = u_{1}-u_{2}, ~~~~v=\rho_{1}-\rho_{2},
\end{equation*}
and $(\omega, v)$ satisfies
\begin{equation}\label{e322}
 \left\{ \begin{array}{ll}
 \omega_{t}+\omega\partial_{x}u_{1}+u_{2}\partial_{x}\omega=\partial_{x}(I-\partial_{x}^{2})^{-1}(v-w), \\
 v_{t}+\omega\partial_{x}\rho_{1}+u_{2}\partial_{x}v+\partial_{x}\omega=0, \\
 \end{array} \right.
 \end{equation}
 with initial data $(\omega(0, x), v(0, x))=0$. Uniqueness follows by following standard energy estimate.
\begin{lemma}\label{lem31}
Let $s > \frac{3}{2}$, for $t \in [0, \widetilde{T}]$ we have the following energy estimate
\begin{equation}\label{e323}
\|\omega(t)\|_{H^{s-1}}+\|v(t)\|_{H^{s-1}}\leq (\|\omega(0)\|_{H^{s-1}}+\|v(0)\|_{H^{s-1}})e^{Kt}=0
\end{equation}
\end{lemma}

\subsection{Continuous dependence}
This subsection follows from a Bona-Smith type argument \cite{Bona}.
Let ${u_{0,n}}$ be a sequence of functions in $H^{s}$ which converge to ${u_{0}}$, ${\rho_{0,n}}$ be a sequence of functions in $H^{s-1}$ which converge to ${\rho_{0}}$. Let $({u_{n}}, \rho_{n})$ and $(u, \rho)$ be the corresponding solutions in $ C([0, T]; H^{s}\times H^{s-1})$. We will show that $({u_{n}}, \rho_{n})\rightarrow (u, \rho)$ in $ C([0, T]; H^{s}\times H^{s-1})$. In other words, for any $\eta > 0$, there exists an $N > 0$ such that for all $n > N$,
$$ \|u-u_{n}\|_{C([0,T]; H^{s})}< \eta,~~~~\|\rho-\rho_{n}\|_{C([0,T]; H^{s-1})}< \eta.$$
We will estimate the above difference by introducing approximate solutions $({u^{\varepsilon}}, \rho^{\varepsilon})$
and $({u_{n}^{\varepsilon}}, \rho_{n}^{\varepsilon})$ which correspond to mollified initial data; i.e. $({u^{\varepsilon}}, \rho^{\varepsilon})$ solves
 \begin{equation}\label{e324}
 \left\{ \begin{array}{ll}
 u^{\varepsilon}_{t}+u^{\varepsilon}\partial_{x}u^{\varepsilon}=\partial_{x}(I-\partial_{x}^{2})^{-1}(\rho^{\varepsilon}-u^{\varepsilon}) \\
 \rho_{t}^{\varepsilon}+u^{\varepsilon}\partial_{x}\rho^{\varepsilon}+\rho^{\varepsilon}\partial_{x}u^{\varepsilon}+\partial_{x}u^{\varepsilon}=0
 \\
 (u^{\varepsilon}, \rho^{\varepsilon})(x, 0)= J_{\varepsilon}(u_{0}(x), \rho_{0}(x)).
 \end{array} \right.
 \end{equation}
 Firstly, by the triangle inequality, we have
 \begin{align}\label{e325}
 &\|U-U_{n}\|_{C([0,T]; H^{s})} \nonumber \\
 &\leq \|U-U^{\varepsilon}\|_{C([0,T]; H^{s})}+\|U^{\varepsilon}-U^{\varepsilon}_{n}\|_{C([0,T];H^{s})}+\|U_{n}-U^{\varepsilon}_{n}\|_{C([0,T];H^{s})} \nonumber \\
 &\leq D_{1}+D_{2}+D_{3},
 \end{align}
 where $U=(u,\rho)$. For convenience, here we just give details on estimate $D_{1}$. Let $ w= u^{\varepsilon}-u^{\varepsilon'}$, $ \upsilon=\rho^{\varepsilon}-\rho^{\varepsilon'}$, which satisfy
\begin{align}\label{e326}
\left\{ \begin{array}{ll}
\omega_{t}+(u\omega+\frac{1}{2}\omega^{2})_{x}=\partial_{x}(I-\partial_{x}^{2})^{-1}(\upsilon-\omega), \\
\upsilon_{t}+(\rho\omega+\upsilon u)_{x}+\partial_{x}\omega=0.
\end{array} \right.
\end{align}
Applying $\Lambda^{s}$ to the first equation and multiplying by $\Lambda^{s}\omega$, we can get
\begin{align}\label{a327}
\frac{1}{2}\frac{d}{dt}\|\omega\|_{H^{s}}^{2}=-\int_{R} \Lambda^{s} (u\omega+\frac{1}{2}\omega^{2})_{x}\Lambda^{s}\omega dx+\int_{R} \Lambda^{s} \partial_{x}(I-\partial_{x}^{2})^{-1}(\upsilon-\omega)\Lambda^{s}\omega dx.
\end{align}
Cauchy-Schwarz inequality and Lemma \ref{lem21} indicate
\begin{align}\label{a328}
\int_{R} \Lambda^{s} \partial_{x}(I-\partial_{x}^{2})^{-1}(\upsilon-\omega)\Lambda^{s}\omega dx
\leq \| \Lambda^{s}\partial_{x}(I-\partial_{x}^{2})^{-1}\upsilon\|_{L^{2}}\|\Lambda^{s}\omega\|_{L^{2}}\leq
\|\upsilon\|_{H^{s-1}}\|\omega\|_{H^{s}}.
\end{align}
Integration by parts and Proposition \ref{p22} yield
\begin{align}\label{a329}
&-\int_{R} \Lambda^{s} (u\omega+\frac{1}{2}\omega^{2})_{x}\Lambda^{s}\omega dx \nonumber\\
&\leq|\int_{R} \Lambda^{s}(u_{x}\omega+u\omega_{x})\Lambda^{s}\omega+\Lambda^{s}(\omega\omega_{x})\Lambda^{s}\omega dx| \nonumber\\
&\leq\|[\Lambda^{s},\omega]u_{x}\|_{L^{2}}\|\Lambda^{s}\omega\|_{L^{2}}+
\|\omega\|_{\infty}\|\Lambda^{s}u_{x}\|_{L^{2}}
\|\Lambda^{s}\omega\|_{L^{2}}+
\|[\Lambda^{s},\omega]\omega_{x}\|_{L^{2}}\|\Lambda^{s}\omega\|_{L^{2}}
+\frac{1}{2}\|\omega_{x}\|_{\infty}\|\Lambda^{s}\omega\|_{L^{2}}^{2}\nonumber\\
&\leq (\|\partial_{x}\omega\|_{\infty}\|\Lambda^{s-1}u_{x}\|_{L^{2}}+
\|\Lambda^{s}\omega\|_{L^{2}}\| u_{x}\|_{\infty})
\|\Lambda^{s}\omega\|_{L^{2}}
+\|\omega\|_{\infty}\|\Lambda^{s}u_{x}\|_{L^{2}}
\|\Lambda^{s}\omega\|_{L^{2}}
+\frac{3}{2}\| \omega_{x}\|_{\infty}\|\Lambda^{s}\omega\|_{L^{2}}\nonumber
\\
&\leq\|\omega\|_{H^{s}}^{2}\| u\|_{H^{s}}+\frac{1}{\varepsilon}\|\omega\|_{H^{\sigma}}\|\omega\|_{H^{s}}\| u\|_{H^{s}}+\|\upsilon\|_{H^{s-1}}\|\omega\|_{H^{s}},
\end{align}
where $\frac{1}{2}<\sigma< s-1$, the last inequality uses Sobolev embedding and Lemma \ref{lem23}.\\
Applying $\Lambda^{s-1}$ to the second equation and multiply by $\Lambda^{s-1}\rho$, we can get
\begin{align}\label{a330}
\frac{1}{2}\frac{d}{dt}\|\upsilon\|_{H^{s-1}}^{2}=-\int_{R}(\Lambda^{s-1}(\rho\omega)_{x}+
\Lambda^{s-1}(vu)_{x})\Lambda^{s-1}\upsilon dx-\int_{R}\partial_{x}\Lambda^{s-1}\omega\Lambda^{s-1}\upsilon dx .
\end{align}
Proposition \ref{p21} and Lemma \ref{lem23} yield
\begin{align}\label{a331}
&-\int_{R}(\Lambda^{s-1}(\rho\omega)_{x}+
\Lambda^{s-1}(\upsilon u)_{x})\Lambda^{s-1}\upsilon dx \nonumber \\
&\leq |\int_{R}[\Lambda^{s-1}\partial_{x}, \omega]\rho\Lambda^{s-1}\upsilon+ \omega\Lambda^{s-1}\partial_{x}\rho \Lambda^{s-1}\upsilon dx | +|\int_{R}[\Lambda^{s-1}\partial_{x}, u]\upsilon\Lambda^{s-1}\upsilon +u\Lambda^{s-1}\partial_{x}\upsilon \Lambda^{s-1}\upsilon dx| \nonumber\\
&\leq \|\omega\|_{H^{s}}\|\rho\|_{H^{s-1}}\|\upsilon\|_{H^{s-1}}
+ \frac{1}{\varepsilon}\|\omega\|_{H^{\sigma}}\|\rho \|_{H^{s-1}}\|\upsilon \|_{H^{s-1}}
+\| u\|_{H^{s}}\|\upsilon\|_{H^{s-1}}^{2}
+\frac{1}{\varepsilon}\| u\|_{H^{\sigma}}\|\upsilon\|_{H^{s-1}}^{2},
\end{align}
where $\frac{1}{2}<\sigma< s-1$.
And
\begin{align}\label{a332}
\int_{R}\partial_{x}\Lambda^{s-1}\omega\Lambda^{s-1}\upsilon dx\leq \|\partial_{x}\Lambda^{s-1}\omega\|_{L^{2}}\|\Lambda^{s-1}\upsilon\|_{L^{2}} \leq \|\omega\|_{H^{s}}\|\upsilon\|_{H^{s-1}}^{2}.
\end{align}
Let $ y_{1}^{2}(t)=\|\omega\|_{H^{s}}^{2}+ \|\upsilon\|_{H^{s-1}}^{2}$,
combining (\ref{a327})-(\ref{a332}) with a standard energy estimate $\|\omega\|_{H^{\sigma}}\leq c\varepsilon^{s-\sigma}$ for $\frac{1}{2}<\sigma< s-1$ (see \cite{Bona}), we have
\begin{align}\label{a333}
\frac{d}{dt}y_{1}(t)\leq K y_{1}(t)+ C(\varepsilon),
\end{align}
where $C(\varepsilon)\rightarrow 0$ as $\varepsilon \rightarrow 0$. The Gronwall's inequality shows
$$ y_{1}(t)= (y_{1}(0) + \frac{C(\varepsilon)}{K} )e^{Kt}- \frac{C(\varepsilon)}{K}.$$
 It's easy to see $ y_{1}(t) < \frac{\eta}{3}$ if choose $\varepsilon$ is small enough.

Since $\| U_{0,n}- U_{0}\|_{H^{s}}\rightarrow 0$, we can get $\| U^{\varepsilon}_{0,n}- U_{0}^{\varepsilon}\|_{H^{s}}\rightarrow 0$. The similar process would be used to deal with $D_{2}$ and $D_{3}$.

\section{ Wave-breaking }

\subsection{Blow~up~scenario}
  We introduce the ordinary equation of the flow generated by $u$
\begin{equation}\label{e41}
  \left\{
  \begin{aligned}
  &\frac{\partial}{\partial t}q(t,x)= u(t,q(t,x)), ~~(t, x)\in [0, T)\times R, \\
 & q(0,x)=x.\\
\end{aligned}
 \right.
\end{equation}
Classical results in the theory of ordinary differential equations imply that there exists a unique solution
$q \in C([0, T)\times R )$ to (\ref{e41}) such that the function $q(t,x)$ is an increasing function with respect to $x$ with
 \begin{equation}\label{e42}
 q_{x}(t,x)= exp ~(\int_{0}^{t}u_{x}(s, q(s,x))ds )> 0 , ~~\forall(t, x)\in [0, T)\times R.
 \end{equation}
 \begin{lemma}\label{lem51}\cite{Guan}
 Let $(u_{0}, \rho_{0})\in H^{s}(R)\times H^{s-1}(R)$, $s> \frac{3}{2}$, $T$ is the maximal existence time
 of the corresponding solution $(u, \rho)$ of (\ref{e31}). Then we have
  \begin{equation}\label{e43}
 (\rho(t, q(t,x))+1)q_{x}(t,x)=  (\rho_{0}(x)+1)
 \end{equation}
\end{lemma}
\begin{lemma}\label{lem52}
 Let $(u_{0}, \rho_{0})\in H^{s}(R)\times H^{s-1}(R)$, $s> \frac{3}{2}$, $T$ is the maximal existence time
 of the corresponding soution $(u, \rho)$ of (\ref{e31}). For any $t \in [0,T)$, we have the following conservations
\begin{equation*}
 \int_{R} u dx=\int_{R}u_{0} dx,~~~~~~~~\int_{R}\overline{\rho} dx=\int_{R}\overline{\rho}_{0} dx.
 \end{equation*}
 Moreover, if $\rho_{0}+1 \geq 0$, we have
 \begin{equation*}
  \|u\|_{L^{2}}\leq \|u_{0}\|_{L^{2}}+\frac{1}{2}\|\overline{\rho}_{0}\|_{L^{1}}t
 \end{equation*}
\end{lemma}
\begin{proof}
   Using the equation (\ref{e11}), integration by parts, we find
\begin{align}\label{a44}
 &\frac{d}{dt}\int_{R}u dx =-\frac{1}{2}\int_{R}(u^{2})_{x}dx+\int_{R}\partial_{x}(I-\partial_{x}^{2})^{-1}\overline{\rho} dx-
 \int_{R}\partial_{x}(I-\partial_{x}^{2})^{-1}udx=0 \nonumber \\
 &\frac{d}{dt}\int_{R}\overline{\rho} dx =-\int_{R}(\overline{\rho} u)_{x}dx=0
\end{align}

  When $\rho_{0}(x)+1\geq 0$, by (\ref{e43}), we can get $\rho(t, x)+1 =\overline{\rho}(t, x) \geq 0$ and $\|\overline{\rho}\|_{L^{1}}=\|\overline{\rho}_{0}\|_{L^{1}}$ for all $[0, T)\times R$.
 \begin{align}\label{a45}
 &\frac{1}{2}\frac{d}{dt}\|u\|_{L^{2}}^{2} =-\int_{R}u^{2}u_{x}dx+\int_{R}\partial_{x}(I-\partial_{x}^{2})^{-1}\overline{\rho} udx-\int_{R}\partial_{x}(I-\partial_{x}^{2})^{-1}uudx\nonumber \\
 &=\int_{R}\partial_{x}(I-\partial_{x})^{-1}\overline{\rho} udx \leq \|\partial_{x}(I-\partial_{x}^{2})^{-1}\overline{\rho} \|_{L^{2}}\|u\|_{L^{2}}
 \leq \frac{1}{2}\|sign(x)e^{-|x|}\|_{L^{2}}\|\overline{\rho}\|_{L^{1}}\|u\|_{L^{2}}\leq  \|\overline{\rho}_{0}\|_{L^{1}}\|u\|_{L^{2}}
\end{align}
 Integrating (\ref{a45}) from 0 to $t$, we can get
 \begin{equation}\label{e46}
 \|u\|_{L^{2}}\leq \|u_{0}\|_{L^{2}}+ \|\overline{\rho}_{0}\|_{L^{1}}t \leq K_{1}(T)
 \end{equation}
\end{proof}
\begin{lemma}\label{lem43}\cite{Boc}
 Let $T > 0$ and $ u\in C^{1}([0, T); H^{2})$. Then for exery $t\in [0, T)$, there exists at least one point $\xi(t)\in R$ with
 \begin{equation}\label{e47}
 m(t)=\inf_{x\in R} u_{x}(t,x)= u_{x}(t, \xi(t))
 \end{equation}
 and the function $m(t)$ is almost everywhere differential on $(0, T)$ with
  \begin{equation}\label{e48}
 \frac{dm}{dt}= u_{tx}(t, \xi(t))~~~a.e.~~(0, T)
 \end{equation}
\end{lemma}
\begin{lemma}\label{lem44}
If $\rho_{0}+1 \geq 0$, we have
\begin{align}
  \sup_{x\in R}u_{x}(t, x)\leq (1+K\|(u_{0},\rho_{0})\|^{2}_{H^{s}\times H^{s-1}})e^{(\|\overline{\rho}_{0}\|_{L^{1}}+ K_{1}(T)+\frac{1}{2})t} \nonumber
\end{align}
\end{lemma}
\begin{proof}
Since $u\in H^{s}$, $s>\frac{3}{2}$, we know
\begin{equation}\label{e49}
\inf_{x\in R}{u_{x}(t,x)}\leq 0, ~~~~t\in [0,T),
\end{equation}
\begin{equation}\label{e410}
\sup_{x\in R}{u_{x}(t,x)}\geq 0.~~~~~t\in [0,T),
\end{equation}
it is sufficient to estimate $\sup_{x\in R}{u_{x}(t,x)}$.

 As the Lemma \ref{lem43} says $\widetilde{M}(t)= u_{x}(t, \xi(t))=\sup_{x\in R}u_{x}(t, x)$, we take the characteristic $q(t, x)$
defined in (\ref{e41}) and choose $x_{1}(t)\in R$ such that
\begin{equation}\label{e411}
q(t, x_{1}(t))= \xi(t).
\end{equation}
 Let  $\gamma(t)=\overline{\rho}(t, q(t, x_{1}))= \overline{\rho}(t, \xi(t))$, along the trajectory $q(t, x_{1})$,
differentiating (\ref{e31}) with respect to $x$, we have
\begin{align}\label{e412}
 &\frac{d}{dt}\widetilde{M }= u_{tx}+uu_{xx}=-\widetilde{M}^{2}+f(t, x_{1})
 \nonumber \\
 &\frac{d}{dt}\gamma = -\gamma \widetilde{M}
\end{align}
 and $f(t,x)$ can be represented as
\begin{align}\label{e413}
f(t,x)=  \partial_{x}^{2}(I-\partial_{x}^{2})^{-1}(\overline{\rho}-u)(t,q).
\end{align}
By lemma \ref{lem52}, we can get
\begin{align}\label{a414}
\mid f(t,x)\mid\leq \frac{1}{2}\|e^{-|x|}\|_{L^{\infty}}\|\overline{\rho}\|_{L^{1}}+ \frac{1}{2}\|e^{-|x|}\|_{L^{1}}\|u\|_{L^{2}}\leq \|\overline{\rho}_{0}\|_{L^{1}}+ K_{1}(T).
\end{align}

Since $\gamma(t,x)$ has the same sign with $\gamma(0,x)= \overline{\rho}_{0}(x)$ for every $x\in R$. In view of Sobolev
imbedding theorem, by $\rho_{0}\in H^{s-1}$, $s-1 > \frac{1}{2}$, we have $\rho_{0}\in C_{0}(R)$
and there exists $R_{0}$ such that $|\rho_{0}|\leq \frac{1}{2}$ for all $|x|\geq R_{0}$. Since $\rho_{0}+1>0$ for all $ x\in R$, it follows that
 \begin{equation}\label{e415}
 \inf_{|x|\leq R_{0}}|\gamma(0,x)|=\inf_{|x|\leq R_{0}}|\rho_{0}+1|>0.
 \end{equation}
 Set $\beta = \rm{min}\{\frac{1}{2}; \inf_{|x|\leq R_{0}}|\gamma(0,x)|\}$, then $|\gamma(0,x)|\geq \beta >0$ for all $x\in R$. Thus
 \begin{equation}\label{e416}
 \gamma(0,x)\gamma(t,x)>0.
 \end{equation}
 We will estimate $\widetilde{M}(t)= \sup_{x\in R}{u_{x}(t,x)}$, (\ref{e410}) imply $\widetilde{M}(t)\geq 0, t\in[0, T)$. We consider the following Lyapunov function
 \begin{equation}\label{e417}
 \widetilde{\omega}(t)=\gamma(0)\gamma(t)+\frac{\gamma(0)}{\gamma(t)}(1+\widetilde{M}^{2}(t))
 \end{equation}
 we have
 \begin{equation*}
 \widetilde{\omega}(t)\geq \gamma(0)\gamma(t);~~\widetilde{\omega}(t)\geq \gamma(t)\widetilde{M}(t).
 \end{equation*}
 Differentiating
\begin{align}\label{e418}
\frac{\partial\widetilde{\omega}}{\partial t}(t)&= -\gamma(0)\gamma(t)\widetilde{M}(t)+\frac{\gamma(0)\widetilde{M}(t)}{\gamma(t)}(1+\widetilde{M}^{2}(t))+ \frac{\gamma(0)}{\gamma(t)}2\widetilde{M}(t)(-\widetilde{M}^{2}(t)+f)\nonumber\\
&= \frac{\gamma(0)\widetilde{M}(t)}{\gamma(t)}(-\gamma^{2}-\widetilde{M}^{2}(t)+2f+1)\nonumber \\
&\leq \frac{2\gamma(0)\widetilde{M}(t)}{\gamma(t)}(|f|+\frac{1}{2})\nonumber \\
&\leq \frac{\gamma(0)}{\gamma(t)}(1+\widetilde{M}^{2}(t))(\|\overline{\rho}_{0}\|_{L^{1}}+ K_{1}(T)+\frac{1}{2})\nonumber \\
&\leq\widetilde{\omega}(t)(\|\overline{\rho}_{0}\|_{L^{1}}+ K_{1}(T)+\frac{1}{2})
\end{align}
An Gronwall inequality shows
\begin{align}\label{a419}
 \widetilde{M}(t)\leq \widetilde{\omega}(t)\leq \widetilde{\omega}(0)e^{(\|\overline{\rho}_{0}\|_{L^{1}}+ K_{1}(T)+\frac{1}{2})t} \leq (1+K\|(u_{0},\rho_{0})\|^{2}_{H^{s}\times H^{s-1}})e^{(\|\overline{\rho}_{0}\|_{L^{1}}+ K_{1}(T)+\frac{1}{2})t}
\end{align}
where we use (\ref{e314}), K is a constant.
\end{proof}

\begin{theorem}\label{th41}
 If $(u_{0}, \rho_{0})\in H^{s}\times H^{s-1}$ ($s >\frac{3}{2}$) and $\rho_{0}+1 \geq 0$, let $T$ be the maximal existence time of the solution $(u,\rho)$ to $(\ref{e31})$. Then the corresponding solution blows up in finite time if and only if
\begin{equation}\label{e420}
 \lim_{t\rightarrow T}\inf_{x\in R}{u_{x}(t,x)}=-\infty
\end{equation}
\end{theorem}
\begin{proof}
Indeed, the sufficiency is obvious due to the embedding theorem, here we mainly pay more attention on the proof of necessity.
 By (\ref{e42}), we know that $q(t, \cdot)$ is an increasing diffeomorphism of $R$ with
 \begin{align}\label{a421}
 & \inf_{x\in R}u_{x}(t, q(t,x))=\inf_{x\in R}u_{x}(t, x)  \nonumber \\
 & \sup_{x\in R}u_{x}(t, q(t,x))=\sup_{x\in R}u_{x}(t, x)
\end{align}
Let
 \begin{align}\label{a422}
 & v(t,x)= u(t, q(t,x)), \nonumber \\
 & h(t,x)= \rho(t, q(t,x)),
\end{align}
 since $u_{x}(t, q(t,x))= v_{x}(t, x)$, by equation (\ref{e31}), we have
\begin{align}\label{e423}
 &\frac{d}{dt}v = \partial_{x}(I-\partial_{x}^{2})^{-1}(h-v) \nonumber \\
 &\frac{d}{dt}h = hv_{x}+v_{x}.
\end{align}
 Applying $\Lambda^{s}$ to the first equation of (\ref{e49}) and multiplying by $\Lambda^{s}v$, then integrating from $R$ we obtain
\begin{align}\label{a424}
 &\frac{1}{2}\frac{d}{dt}\|\Lambda^{s}v\|_{L^{2}}^{2} =\int_{R}\Lambda^{s}\partial_{x}(I-\partial_{x}^{2})^{-1}(h-v)\Lambda^{s}vdx
\leq \|h\|_{H^{s-1}}\|v\|_{H^{s}}+\|v\|_{H^{s}}^{2}
\end{align}
Apply $\Lambda^{s-1}$ to the second equation of (\ref{e49}) and multiply by $\Lambda^{s-1}\gamma$ and integrate from $R$ to obtain
\begin{align}\label{a425}
 \frac{1}{2}\frac{d}{dt}\|\Lambda^{s-1}h\|_{L^{2}}^{2} &=\int_{R}\Lambda^{s-1}(hv_{x})\Lambda^{s-1}h dx +\int_{R}\Lambda^{s-1}\partial_{x}v\Lambda^{s-1}h dx \nonumber \\
&\leq \int_{R}\|\Lambda^{s-1}(hv_{x})\|_{L^{2}}\|\Lambda^{s-1}h\|_{L^{2}}dx +\int_{R}\Lambda^{s-1}\partial_{x}v\Lambda^{s-1}h dx\nonumber \\
&\leq \|hv_{x}\|_{H^{s-1}}\|h\|_{H^{s-1}}+\|v\|_{H^{s}}\|h\|_{H^{s-1}}\nonumber \\
&\lesssim (\|h\|_{\infty}\|v\|_{H^{s}}+\|h\|_{H^{s-1}}\|v_{x}\|_{\infty})\|h\|_{H^{s-1}}+\|v\|_{H^{s}}\|h\|_{H^{s-1}}
\end{align}
Adding (\ref{a424}) to (\ref{a425}), we have
\begin{align}\label{a426}
 &\frac{1}{2}\frac{d}{dt}\|\Lambda^{s}v\|_{L^{2}}^{2}+\frac{1}{2}\frac{d}{dt}\|\Lambda^{s-1}h\|_{L^{2}}^{2} \nonumber \\
&\lesssim \|h\|_{H^{s-1}}\|v\|_{H^{s}}+\|v\|_{H^{s}}^{2}+(\|h\|_{\infty}\|v\|_{H^{s}}+\|h\|_{H^{s-1}}\|v_{x}\|_{\infty})
\|h\|_{H^{s-1}}+\|v\|_{H^{s}}\|h\|_{H^{s-1}}
\nonumber \\
&\lesssim(\| h\|_{\infty}+\| v_{x} \|_{\infty}+1)(\|h\|_{H^{s-1}}^{2}+\|v\|_{H^{s}}^{2})
\end{align}
By the Gronwall's inequality and (\ref{a426}), we obtain
\begin{equation}\label{e427}
\|h\|_{H^{s-1}}^{2}+\|v\|_{H^{s}}^{2}\lesssim exp \int_{0}^{t}(\|\rho\|_{\infty}+
\| u_{x}\|_{\infty}+1)(s,q)ds
\end{equation}

 When $\|h\|_{H^{s-1}}^{2}+\|v\|_{H^{s}}^{2}\rightarrow \infty $ with $T < \infty$, by (\ref{e427}) we can get $\| \rho \|_{\infty}\rightarrow \infty$ or $\| u_{x}\|_{\infty}\rightarrow \infty $.\\
 {\bf case 1}: If $\| u_{x}\|_{\infty}\rightarrow \infty $, combining with the upper bound found in Lemma \ref{lem44} we finish the proof.\\
 {\bf case 2}: If $\|\rho \|_{\infty}\rightarrow \infty$, i.e. $\|\overline{\rho} \|_{\infty}\rightarrow \infty$, then $u_{x} < 0$. Otherwise, by(\ref{e42}) and (\ref{e43}), we can get
\begin{equation}\label{e428}
 \|\overline{\rho} \|_{\infty}=\|\rho(t, q(t,x))+1)\|_{\infty}\leq |q_{x}^{-1}(t,x)|~\| \rho_{0}(x)+1\|_{\infty}\leq
 \| \rho_{0}(x)+1\|_{\infty},
\end{equation}
which is a contradiction. From (\ref{e42}) and (\ref{e428}), we can deduce $\lim_{t\rightarrow T}\inf_{x\in R}{u_{x}(t,x)}=-\infty$.
 \end{proof}

\subsection{Sufficient conditions of wave breaking}
In this subsection, we will show some sufficient conditions on wave breaking.
\begin{theorem}\label{th42}
 Let $(u_{0}, \rho_{0})\in H^{s}\times H^{s-1}$ and $s>\frac{3}{2}$, $T$ is the maximal existence time
 of the corresponding solution $(u, \rho)$ of (\ref{e31}). Assume that there exists $x_{0}\in R$ such that $u_{0}'(x_{0})\leq -(1+\varepsilon)\widetilde{K}$, where $\widetilde{K }$ is defined in (\ref{a434}), then $T$ is finite and the slope of $u$ tends to negative infinity as $t$ tends to $T$.
 \end{theorem}
\begin{proof}
Define $m(t)= \inf_{x\in R} u_{x}(t, q(t,x_{0}))$ and $\gamma(t)= \rho(t, q(t,x_{0}))+1$.
By (\ref{e31}), we have
\begin{align}\label{e429}
 &\frac{d}{dt}m =-m^{2}+ (p\ast\gamma-\gamma)(t,x_{0})+(-p\ast u+u)(t, q(t,x_{0})),
 \nonumber \\
 &\frac{d}{dt}\gamma = -\gamma m,
\end{align}
where $p(x)=\frac{1}{2}e^{-|x|}$. Let
\begin{align}\label{e52}
f= \partial_{x}^{2}(I-\partial_{x}^{2})^{-1}\overline{\rho}(t,q)-\partial_{x}(I-\partial_{x}^{2})^{-1}u_{x}(t,q),
\end{align}

If we can find $T < \infty$ such that $\lim_{t\uparrow T} u_{x}\rightarrow -\infty$, blow up occurs.

 Otherwise, for any $T < \infty$ such that $ \|u_{x}\|_{\infty}\leq C(T)$, we have
\begin{equation}\label{e431}
\|\overline{\rho}(t,x)\|_{\infty}\leq (\|\rho_{0}\|_{\infty}+1)e^{C(T)t}\leq K_{2}(T).
\end{equation}
and
\begin{align}\label{e432}
\mid f\mid\leq \|p\|_{L^{1}}\|\overline{\rho}\|_{L^{\infty}}+ \|p\|_{L^{1}}\|u_{x}\|_{L^{\infty}}
\leq K_{2}(T)+C(T)
\end{align}
We obtain the relation
\begin{align}\label{a433}
 \frac{d}{dt}m \leq -m^{2}+ K_{2}(T)+C(T)
\end{align}
Taking $J(T_{1})=\sqrt{K_{2}(T_{1})+C(T_{1})}$, it satisfies
\begin{align}\label{a434}
 T_{1}J^{2}(T_{1})=log({1+\frac{1}{\varepsilon}}).
\end{align}
Since $T_{1}J^{2}(T_{1})$ is a continuous function of $T_{1}$ and $log({1+\frac{1}{\varepsilon}})\in R^{+}$,  for fixed $\varepsilon$, the function (\ref{a434}) has solution $T_{1}$.

 In views of (\ref{a434}), it follows for all $t\in[0, T_{1}]\cap[0, T)$ that
\begin{align}\label{a435}
 \frac{d}{dt}m \leq -m^{2}+ J(T_{1})^{2}
\end{align}
By the assumption of the theorem, we find
\begin{align}\label{a436}
 m(0)< - (1+\varepsilon)J(T_{1})
\end{align}
which implies that
\begin{align}\label{a437}
 1 < \frac{m(0)-J(T_{1})}{m(0)+J(T_{1})}< 1+\frac{2}{\varepsilon}
\end{align}

Since $t=0$, $m'(0)< 0$ from (\ref{a435}) and (\ref{a436}), standard argument of continuity shows
$$m(t)< -(1+\varepsilon)J(T_{1}),$$
 for all $t\in[0, T_{1}]\cap[0, T)$.

and
\begin{align}\label{a438}
 log\frac{m(t)+ J(T_{1})}{m(t)-J(T_{1})}\geq log \frac{m(0)+ J(T_{1})}{m(0)-J(T_{1})}+ 2J(T_{1})t
\end{align}

since $ 0 <\frac{m(0)+ J(T_{1})}{m(0)-J(T_{1})} < 1$, by (\ref{a437}) and (\ref{a438}), we have
\begin{align}\label{a439}
 \widetilde{T}\leq \frac{log \frac{m(0)-J(T_{1})}{m(0)+J(T_{1})}}{2J(T_{1})} < \frac{log (1+\frac{1}{\varepsilon})}{J^{2}(T_{1})}=T_{1}< T
\end{align}
such $\lim_{t\uparrow\widetilde{T}}m(t)=-\infty $, which is a contradictionㄛwhich concludes the proof of the theorem.
\end{proof}
\begin{theorem}\label{th43}
  Let $(u_{0}, \rho_{0})\in H^{s}(R)\times H^{s-1}(R)$ and $s> \frac{3}{2}$, $T$ is the maximal existence time
 of the corresponding soution $(u, \rho)$ of (\ref{e31}), satisfies $\rho_{0}+1\geq 0$ and
\begin{equation*}
 \inf_{x\in R}u_{0}'(x)+\sup_{x\in R}u_{0}'(x)\leq -2
\end{equation*}
 without loss of generality, here, we assume $\|\overline{\rho}_{0}\|_{L^{1}}=1$,
then we observe wave-breaking for the solution of (\ref{e31}) with initial data $u_{0}$
\end{theorem}
 \begin{proof}
 Let's introduce the
 \begin{equation*}
 m(t)=\inf_{x\in R}u_{x}(t, x)~~~~~~~~~~~~M(t)=\sup_{x\in R}u_{x}(t, x),
 \end{equation*}
 applying lemma \ref{lem43} with appropriate $\xi_{1}(t)$ and $\xi_{2}(t)$ to these functions, we have
  \begin{equation*}
 m(t)=u_{x}(t, \xi_{1}(t))~~~~~~~~~~~~M(t)=u_{x}(t, \xi_{2}(t))
 \end{equation*}
 Differentiating the first equation of (\ref{e31}) with respect to x and evaluating the result equation at $\xi_{1}(t)$ and $\xi_{2}(t)$, we get that, for a.e. $t \in [0, T)$,
 \begin{align}\label{a440}
  m'(t)+m^{2}=\partial_{x}^{2}(I-\partial_{x}^{2})^{-1}(\rho-u)(t,\xi_{1}(t)),\nonumber  \\
  M'(t)+M^{2}=\partial_{x}^{2}(I-\partial_{x}^{2})^{-1}(\rho-u)(t,\xi_{2}(t)).
 \end{align}
 By observing
 \begin{align}\label{a441}
 \partial_{x}(I-\partial_{x}^{2})^{-1}u_{x}=
 -\frac{e^{-x}}{2}\int_{-\infty}^{x}e^{y}u_{y}dy+\frac{e^{x}}{2}\int_{x}^{+\infty}e^{-y}u_{y}dy,
 \end{align}
 here, $\| \overline{\rho}_{0}\|_{L^{1}}=1$ and
 \begin{align}\label{a442}
 |\partial_{x}^{2}(I-\partial_{x}^{2})^{-1}\rho|=|\int_{R}e^{-|x-y|}sign(x-y)\rho dy|\leq \|\overline{\rho}\|_{L^{1}}=\|\overline{\rho}_{0}\|_{L^{1}},
 \end{align}
 we can get
 \begin{align}\label{a443}
 -m'\leq-m^{2}+\frac{1}{2}(M-m)+1,\nonumber \\
 -M'\leq-M^{2}+\frac{1}{2}(M-m)+1.
 \end{align}
 Since
 \begin{align}\label{a444}
 -m'=-m^{2}+\frac{1}{2}(M-m)+1= -(m+\frac{1}{2})^{2}+\frac{1}{2}(m+M)+1,
 \end{align}
 summing up, we get
 \begin{align}\label{a445}
 \frac{d}{dt}(m+M)&\leq -m^{2}+(M+2)^{2}-2M^{2}-4M-2 \nonumber \\
 &\leq -m^{2}+(M+2)^{2}=(M-m+2)(M+m+2).
 \end{align}
 For $m(0)+M(0)+2 \leq 0$ at time $t=0$, (\ref{a445}) and standard argument of continuity shows
\begin{align}\label{a446}
m(t)+M(t)+2\leq 0 ~~{\rm for~~all~~} t\in[0, T).
\end{align}
Denoting $\widetilde{m}(t)=m(t)+\frac{1}{2}$, $t\in[0, T)$, by (\ref{a446}), we see that $\widetilde{m}(0)< 0$ and
\begin{align}\label{a447}
\frac{d}{dt}\widetilde{m}\leq -\widetilde{m}^{2} ~~~a.e~~(0, T).
\end{align}
Integrating (\ref{a447}) yields
\begin{align}\label{a448}
\frac{1}{\widetilde{m}(t)}\geq \frac{1}{\widetilde{m}(0)}+t,~~~t\in[0, T),
\end{align}
so that $\widetilde{m}(t)\rightarrow -\infty$ before $t$ reaches $\frac{1}{|\widetilde{m}(0)|}$, thus proving that the wave breaks in finite time.
 \end{proof}

\section{Travelling waves}
In this Section, we are devoted to establish analytically the existence of travelling solutions of the system (\ref{e11}). Firstly we assume
 \begin{equation}\label{e51}
 \left\{ \begin{array}{ll}
 u(x,t)= \phi(y), \\
  \overline{\rho}(x,t)= \psi(y),
 \end{array} \right.
 \end{equation}
 where $y=x-ct$, $c > 0$ being the speed of travelling wave. Thus, the system (\ref{e11}) would be transformed into
 \begin{equation}\label{e52}
 \left\{ \begin{array}{ll}
 -c\phi_{y}+ \phi\phi_{y}= \partial_{y}(I-\partial_{y}^{2})(\psi-\phi), \\
 -c\psi_{y}+\phi\psi_{y}+\phi_{y}\psi=0.
 \end{array} \right.
 \end{equation}
 Integrating (\ref{e52}) from $0$ to $y$, we can get
 \begin{equation}\label{e53}
 \left\{ \begin{array}{ll}
 -c\phi+ \frac{1}{2}\phi^{2}-(I-\partial_{y}^{2})^{-1}(\psi-\phi)=A, \\
 -c\psi+\phi\psi= B,
 \end{array} \right.
 \end{equation}
 for some real constant $A$, $B$. Here we shall consider the case when $ B=cA\geq 0$. In fact, if $ B=cA= 0$ holds, (\ref{e53}) will be reduced to
 \begin{align}\label{a54}
 -c\phi+ \frac{1}{2}\phi^{2}+(I-\partial_{y}^{2})^{-1}\phi=0,
\end{align}
 due to the fact $\phi < c$, which is corresponding to the travelling waves form of the $Fornberg-Whitham$ equation.

 Substituting the second equation into the first equation in (\ref{e53}) yields
\begin{align}\label{a55}
 -c\phi+ \frac{1}{2}\phi^{2}-(I-\partial_{y}^{2})^{-1}(\frac{cA}{\phi-c}-\phi)-A =0,
\end{align}
 which is equivalent to
 \begin{align}\label{a56}
 \phi- \frac{1}{2c}\phi^{2}+\frac{1}{2c}e^{-|y|}\ast(\frac{cA}{\phi-c}-\phi)+\frac{A}{c} =0.
\end{align}

Then the Crandall-Rabinowitz local bifurcation theorem (see \cite{Deng1}) would be used to prove the existence of travelling wave, here we state it again for our purposes.
\begin{lemma}\label{lem61}
Let $W$ be a Banach space and $F\in C^{k}(R\times W, W)$ with $k\geq2$ satisfy

(1)$F(c,0)=0$ for all $c\in R^{+}$;

(2)$L=\partial_{\phi}F(c^{*},0)\in L(W,W)$ is a Fredholm operator of index zero with $ker{L}$ one-dimensional;

(3)$[\partial^{2}_{c\phi}F(c^{*},0)](1,ker(L))\notin R(L)$ holds, where $\partial^{2}_{c\phi}F(c^{*},0)=\partial_{c}[\partial_{\phi}F(c,0)]|_{c=c^{*}}\in L(R\times W, W)$;\\
Then there exists $\varepsilon>0$ and a continuous bifurcation curve $\{(c_{s},\phi_{s}): |s|<\varepsilon \}$ with $c_{s}|_{s=0}=c^{*}$, where $c^{*}$ is a bifurcation point, such that $\phi_{0}$ is the trivial solution of (\ref{a56}), and $\{ \phi_{s}: s\neq0\}$ is a family of nontrivial solutions with corresponding wave speeds $\{c_{s}\}_{s}$. Moreover, $dist(\phi_{s}, ker(L))=o(s)$ in $W$.
\end{lemma}

\begin{theorem}\label{th51}
For a given $L > 0$, there exists a local bifurcation curve of $2L$-periodic, even and continuous solutions $\phi \in C[-L, L]$ of (\ref{e53}).
\end{theorem}
\begin{proof}
 For verifying above three items in Crandall-Rabinowitz theorem, we define firstly the following function
 \begin{align}\label{a57}
 F(c,\phi)= \phi- \frac{1}{2c}\phi^{2}+\frac{1}{2c}e^{-|y|}\ast(\frac{cA}{\phi-c}-\phi)+\frac{A}{c},
\end{align}
It is easy to see $F(c, 0)=0$ for all $ c > 0$. Then we take the linearized equation
\begin{align}\label{a58}
 Lu(y):= u(y)-\frac{A+c}{2c^{2}}e^{-|y|}\ast u(y)=0
\end{align}
into consideration. Without loss of generality, we assume $u(y)$ is 2$\pi$-periodic, even and continuous function,
which gives $u \in L^{\infty}(R)$. Then taking fourier transform on (\ref{a58}), we have
\begin{align}\label{a59}
(1-\frac{A+c}{c^{2}(1+k^{2})})\widehat{u}(k)=0
\end{align}
in the sense of distributions.

Now, we assume that $k_{0}=k_{0}(c, A)> 0$, such that $(1+k_{0}^{2}) c^{2}= A + c$. By solving a quadratic equation
$ -c^{2}+ c+ A =0$, we know that
 \begin{equation}\label{e510}
 \left\{ \begin{array}{ll}
 \widehat{u}(k)=0 \rm {~for ~all~ k}, &  \textrm{$ \rm{if} \ c> \frac{1}{2}+\sqrt{A+\frac{1}{4}}$}, \\
 $\rm {the~support~of}$ ~\widehat{u}~ $\rm{is~in~ }$ \{\pm k_{0}\},& $\rm{if}$ ~0 <c < \frac{1}{2}+\sqrt{A+\frac{1}{4}}, \\
 $\rm {the~support~of}$ ~\widehat{u}~ $\rm{is~in~ }$ \{ 0\},& $\rm{if}$ ~ c =\frac{1}{2}+\sqrt{A+\frac{1}{4}}.
\end{array} \right.
 \end{equation}
 Thus the nontrivial even periodic solutions of linear problem (\ref{a58}) are given by
 \begin{equation}\label{e511}
 \left\{ \begin{array}{ll}
 u(y)=C, & c =\frac{1}{2}+\sqrt{A+\frac{1}{4}},\\
 u(y)=Ccos(k_{0}y),& c <\frac{1}{2}+\sqrt{A+\frac{1}{4}},
 \end{array} \right.
 \end{equation}
 where $C\in R$ be any nonzero constant. In our analysis, we expect to find the $2\pi$-periodic solutions with physical
 meaning, thus the constant solutions are discarded and let $k_{0}=1$ in (\ref{e511}), which leads to $c =\frac{1}{4}+\sqrt{\frac{8A+1}{16}}$ less than $\frac{1}{2}+\sqrt{A+\frac{1}{4}}$. Then we define $c^{\star}: =\frac{1}{4}+\sqrt{\frac{8A+1}{16}} $, and we obtain
 \begin{align}\label{a512}
 ker(L)=Ccosy,~~~~~~~\textrm{$\rm{with}$}~~~dim (ker(L))=1.
 \end{align}
For above $u(y)$, we can write
\begin{align}\label{a513}
 u(y)\sim \frac{1}{2\pi}\sum_{k\in Z}\widehat{u}(k)e^{iyk},
\end{align}
where $\widehat{u}(k)= \int_{-\pi}^{\pi}u(y)e^{-iyk}dy $, for $u \in L^{p}([-\pi, \pi]),~p \geq 1$.

In fact, the Carleso-Hunt theorem  in \cite{Jor} guarantees that the series (\ref{a513}) converges to $u(y)$ almost everywhere.
 The evenness of $u(y)$ will ensure
\begin{align}\label{a514}
 u(y)= \frac{1}{2\pi}\widehat{u}(0)+\frac{1}{\pi}\sum_{k=1}^{\infty}\widehat{u}(k)cos(ky)~~~~a.e~~~~\rm{on}~[-\pi, \pi].
\end{align}
Since $e^{-|y|}\in L^{1}(R)$, we can write the integral
 \begin{align}\label{a515}
  \int_{-\infty}^{+\infty} e^{-|y-z|}u(z)dz= \sum_{k=-\infty}^{+\infty}\int_{-\pi}^{\pi}e^{-|y-z+2k\pi|}u(z)dz=
  \int_{-\pi}^{\pi}(\sum_{k=-\infty}^{+\infty}e^{-|y-z+2k\pi|})u(z)dz
 := \int_{-\pi}^{\pi}A(y-z)u(z)dz
 \end{align}
 The Minkowski's inequality shows that $ A \in L^{p}([-\pi,\pi])$ for $p\geq 1$, and the definition of $A(y)$ implies
 that it's $2\pi$-periodic, even and continuous. Therefore, $A(y)$ can be writen:
 \begin{align}\label{a516}
 A(y)= \frac{1}{2\pi}\widehat{A}(0)+\frac{1}{\pi}\sum_{k=1}^{\infty}\widehat{A}(k)cos(ky)~~a.e~~\rm{on}~[-\pi, \pi].
\end{align}
The periodic problem is given by the same multiplier on the line, so we have
 \begin{align}\label{a517}
 e^{-|y|}\ast u(y)= \frac{1}{2\pi}\widehat{A}(0)\widehat{u}(0)+\frac{1}{\pi}\sum_{k=1}^{\infty}\widehat{A}(k)\widehat{u}(k)cos(ky)
 \end{align}
 holds almost everywhere on $[-\pi, \pi]$.
 Next, looking for $2\pi$-periodic, even and continuous solutions, we introduce the following Banach space
 \begin{align}\label{a518}
 W:=\{ u(y)= \frac{1}{2\pi}\widehat{u}(0)+\frac{1}{\pi}\sum_{k=1}^{\infty}\widehat{u}(k)cos(ky)\mid \|u\|:=
 \frac{1}{2\pi}|\widehat{u}(0)|+\frac{1}{\pi}\sum_{k=1}^{\infty}|\widehat{u}(k)|< \infty\},
 \end{align}
 and (\ref{a58}), (\ref{a514})and (\ref{a517}) imply
 \begin{align}\label{a519}
  Lu(y)= \frac{1}{2\pi}\widehat{u}(0)(1-\frac{A+c}{2c^{2}}\widehat{A}(0))+\frac{1}{\pi}\sum_{k=1}^{\infty}\widehat{u}(k)
  (1-\frac{A+c}{2c^{2}}\widehat{A}(k))cos(ky)
 \end{align}
 holds almost everywhere on $[-\pi, \pi]$.
 From the definition of $A(y)$, we have
 \begin{align}\label{a520}
  &\widehat{A}(k)=\int_{-\pi}^{+\pi}\sum_{k=-\infty}^{+\infty}e^{-|y-z+2j\pi|}e^{-iky}dy=
  \sum_{k=-\infty}^{+\infty}\int_{-\pi}^{\pi}e^{-|y+2j\pi|}e^{-i(ky+2j\pi)}dy
 = \int_{-\pi}^{\pi}e^{-|y|}e^{-iky}dy  \nonumber \\
 &=\widehat{e^{-|y|}}(k)=\frac{2}{1+k^{2}}
 \end{align}
 Thus it's easy from (\ref{a520}) to see
  \begin{align}\label{a521}
  \lim_{|k|\rightarrow \infty}\widehat{A}(k)=0,
  \end{align}
  which is consistent with Riemann-Lebesgue Lemma.
  Then (\ref{a519}) and (\ref{a521}) indicate
  \begin{align}\label{a522}
  \|Lu\|\leq (1+\frac{A+c}{2c^{2}}\max_{Z}\widehat{A}(k))\|u\|,
 \end{align}
 such that $L\in \mathcal{L}(W, W)$.
 Based on (\ref{a520}), we find that
  \begin{align}\label{a523}
  \widehat{A}(1)=1=\frac{2c^{\star}}{A+c^{\star}},
 \end{align}
  \begin{align}\label{a524}
   \widehat{A}(k)\neq\frac{2c^{\star}}{A+c^{\star}}, k\neq 1,
 \end{align}
 the equality (\ref{a523}) uses the definition of $c^{\star}$, and inequality (\ref{a524}) use the monotonicity of
 $\widehat{A}(k)$ on $N$.
 From (\ref{a512}), (\ref{a519}), (\ref{a523}) and (\ref{a524}), we know that
 \begin{align}\label{a525}
  W = ker(L)\oplus R(L),
 \end{align}
 that is to say
 \begin{align}\label{a526}
   dim(ker(L))= dim(W \backslash R(L))=1.
 \end{align}

 Finally, we take derivative with respect to bifurcation parameter $c$ on (\ref{a58}), and evaluate at $c^{\star}$ is
 \begin{align}\label{a527}
  (\partial_{c}L|_{c=c^{\star}})(1,u(y))=\frac{2A+c^{\star}}{2{c^{\star}}^{3}}e^{-|y|}\ast u(y).
 \end{align}
 By (\ref{a517}), we have that
 \begin{align}\label{a528}
 (\partial_{c}L|_{c=c^{\star}})(1,u(y))= M (\frac{1}{2\pi}\widehat{A}(0)\widehat{u}(0)+\frac{1}{\pi}\sum_{k=1}^{\infty}\widehat{A}(k)\widehat{u}(k)cos(ky)),
 \end{align}
 where $M=\frac{8+64A+8\sqrt{8A+1}}{(1+\sqrt{8A+1})^{3}}$ is a fixed constant.\\
 Therefore, by the same argument as (\ref{a522}), we would obtain from (\ref{a528}) that
 \begin{align}\label{a529}
 \|(\partial_{c}L|_{c=c^{\star}})(1,u(y))\|\leq M \max_{Z}{\widehat{A}(k)}\|u\|=2M\|u\|,
 \end{align}
 the equality is due to (\ref{a520}), which indicates $\partial_{c}L|_{c=c^{\star}}\in \mathcal{L}(R\times W, W)$.
 In particular, we choose $u(y)= ker(L)=Ccos(y)$ in (\ref{a528}), then
 \begin{align}\label{a530}
 (\partial_{c}L|_{c=c^{\star}})(1,ker(L))\cap R(L)= ker(L)\cap R(L)=\emptyset,
 \end{align}
 due to the support of $\mathcal{F}(cos(y))$ is in $\{\pm 1\}$ and (\ref{a525}).

 Up to now, we finish the proof.
\end{proof}

\section*{Acknowledgement}
The authors acknowledge the support of the National Natural Science Foundation of China (No.11571057).

\section*{References}

\bibliography{mybibfile}

\end{document}